\documentclass[leqno]{amsart}
\usepackage{srctex,color}

\renewcommand{\Re}{\operatorname{Re}}
\newcommand{\Tr}{\mathrm{Tr}}

\newcommand{\dd}{\mathrm{d}}

\newcommand{\R}{\mathbb R}
\newcommand{\E}{\mathbf E}

\newcommand{\TT}{{\mathcal T}}

\newcommand{\inner}[3][]{(#2,#3 )_{#1}}
\newcommand{\HS}{\text{HS}}
\newcommand{\T}[1]{\tilde{#1}}

\newtheorem{theorem}{Theorem}[section]
\newtheorem{lemma}[theorem]{Lemma}
\newtheorem{proposition}{Proposition}[section]
\newtheorem{hypothesis}{Assumption}
\newtheorem{corollary}[theorem]{Corollary}

\theoremstyle{definition}

\theoremstyle{remark}
\newtheorem{remark}[theorem]{Remark}

\numberwithin{equation}{section}

\begin{document}
\title[Weak approximation for linear stochastic Volterra equations]{Weak convergence of a fully discrete approximation of a linear stochastic evolution equation with a positive-type memory term}

\author{Mih\'aly Kov\'acs}
\address{Department of Mathematics and Statistics, University of Otago, PO Box 56, Dunedin, 9054, New Zealand.}
\email{mkovacs@maths.otago.ac.nz}

\author{Jacques Printems}
\address{Laboratoire d'Analyse et de Math\'ematiques Appliqu\'ees CNRS UMR 8050, Universit\'e Paris--Est, 61, avenue du G\'en\'eral de Gaulle, 94 010 Cr\'eteil, France.}
\email{printems@u-pec.fr}

\date{}



\maketitle

\begin{abstract}
In this paper we are interested in the numerical approximation of the marginal distributions of the Hilbert space valued solution of a stochastic Volterra equation driven by an additive Gaussian noise. This equation can be written in the abstract It\^o form as
$$
\dd X(t) + \left ( \int_0^t b(t-s) A X(s) \, \dd s \right ) \, \dd t = \dd W^{_Q}(t),~t\in (0,T]; ~ X(0) =X_0\in H,
$$
\noindent where $W^Q$ is a $Q$-Wiener process on the Hilbert space $H$ and where the time kernel $b$ is the locally integrable potential $t^{\rho-2}$, $\rho \in (1,2)$, or slightly more general. The operator $A$ is unbounded, linear, self-adjoint, and positive on $H$. Our main assumption concerning the noise term is that $A^{(\nu- 1/\rho)/2} Q^{1/2}$ is a Hilbert-Schmidt operator on $H$ for some $\nu \in [0,1/\rho]$.
The numerical approximation is achieved via a standard continuous finite element method in space (parameter $h$) and an implicit Euler scheme and a Laplace convolution quadrature in time (parameter $\Delta t=T/N$).
We show that for $\varphi : H\rightarrow \R$ twice continuously differentiable test function with bounded second derivative,
$$
| \E \varphi(X^N_h) -  \E \varphi(X(T)) | \leq C \ln \left( \frac{T}{h^{2/\rho} + \Delta t} \right ) (\Delta t^{\rho \nu} + h^{2\nu}),
$$
\noindent for any $0\leq \nu \leq 1/\rho$. This is essentially twice the rate of strong convergence under the same regularity assumption on the noise.
\end{abstract}

\section{Introduction}

Let $H=L^2(\mathcal{O})$ be the real separable Hilbert space of square integrable functions on some bounded domain $\mathcal{O}$ of $\R^d$, $d\geq 1$, with smooth or convex polygonal boundary equipped with the usual inner product denoted by $(\cdot,\cdot)$ and induced norm $\|\cdot\|$.
For $T>0$, we consider the following stochastic Volterra type equation written in the abstract It\^o form as
\begin{equation}\label{eq:stovolterra}
\dd X(t) + \left ( \int_0^t b(t-s) A X(s) \, \dd s \right ) \, \dd t = \dd W^{_Q}(t), \quad t\in (0,T]; \quad X(0)=X_0,
\end{equation}
\noindent where $-A=\Delta$ is the Dirichlet Laplacian with domain $D(A)=H^2(\mathcal{O})\cap H^1_0(\mathcal{O})$ and $W^{_Q}$ is an $H$-valued Wiener process on the probability space $(\Omega, \mathcal{F}, \mathbf{P})$ endowed with the normal filtration $\{{\mathcal F}_t\}_{t\ge 0}$ generated by $W^{_Q}$ with possibly unbounded covariance operator $Q$. We note that, strictly speaking, the process $W^{_Q}$ is $H$-valued if and only if $Q$ is a trace class operator. The initial condition $X(0)=X_0$ is $H$-valued and $\mathcal{F}_0$-measurable.
The convolution kernel $b$ is given by
\begin{equation}\label{eq:beta}
b(t) = t^{\rho-2}/\Gamma(\rho-1), \quad 1 < \rho <2,
\end{equation}
or could be somewhat more general which is made precise later in Section \ref{ch:pre}. Such equations are called stochastic Volterra equations and can be used in the modeling of diffusion of heat in materials with memory or in viscoelasticity (see \cite{CDaPP,McLThomee93} and references therein). Because of the weak singularity of the kernel $b$ at 0, the deterministic equation exhibit certain smoothing characteristics similar to that of parabolic type evolution equations.


We study the numerical approximation of  $\{X(t)\}_{t\in [0,T]}$ by an implicit Euler scheme and a Laplace transform convolution quadrature in time together with a finite element method in space. Let $N\geq 1$ and $\Delta t = T/N$. We set $t_n = n\Delta t$, $n=0,\dots, N$. Let $\{\TT_h\}_{0<h<1}$ denote a family of triangulations of $\mathcal O$, with mesh size $h>0$ and consider
   finite element spaces $\{ V_h \}_{0<h<1}$, where $V_h\subset H^1_0(\mathcal{O})$ consists of continuous piecewise linear functions vanishing at the boundary of $\mathcal{O}$.  Let $X^n_h \in V_h$ be the numerical approximation of $X(t_n)$ defined via the difference equations
\begin{equation}\label{eq:alg}
(X^n_h,v_h) - (X^{n-1}_h,v_h) + \Delta t \sum_{k=1}^n \omega_{n-k} (\nabla X^k_h,\nabla v_h) = (w^n,v_h),
\end{equation}
\noindent for any $n\geq 1$, with the initial condition
$$
(X^0_h,v_h) = (X_0,v_h),
$$
\noindent for any $v_h\in V_h$, where we have set $w^n = W^{_Q}(t_n) - W^{_Q}(t_{n-1})$.

 Our specific choice of the weights $\{\omega_k\}_{k\geq 0}$ is stemming from the deterministic framework of \cite{lubich88,lubich88II}. Indeed, it can be easily seen that most of the
qualitative properties
of the solution of (\ref{eq:stovolterra}) with $Q=0$ depend heavily on the way the frequencies of the time kernel $b$ are
distributed. For example, in the case where $b$ is a Dirac mass at 0, we formally recover the heat equation and if
$b$ is regular enough, we recover the wave equation.  For that reason, the weights $\{\omega_k\}_{k\geq 0}$ in (\ref{eq:alg}) have been chosen
such as to mimic, at the level of the backward Euler scheme, the spectral properties of the time kernel $b$. Using the Laplace transform
$\widehat b$ of $b$, these weights can be obtained via the relation
\begin{equation} \label{eq:weight}
\widehat b\left ( \frac{1 - z}{\Delta t} \right ) = \sum_{k\geq 0} \omega_k z^k, \quad |z|<1.
\end{equation}

Introducing the "discrete Laplacian"
\begin{equation}\label{def:Ah}
  A_{h}:V_h\to V_h,   \quad
    \inner{A_{h} \psi}{ \chi} = \inner{ \nabla \psi}{\nabla \chi},\quad \psi,\chi \in V_h,
\end{equation}
and the orthogonal projector
$$
    P_{h}: H \to V_h,\quad
   \inner{P_{h} f}{ \chi} = \inner{ f} {\chi},\quad \chi \in V_h.
$$
we rewrite \eqref{eq:alg} in the operator form as
\begin{equation} \label{eq:full_scheme}
X^n_{h} - X^{n-1}_h + \Delta t \left ( \sum_{k=1}^{n} \omega_{n-k}\,  A_h X^k_h \right ) = P_hw^n, \quad n\geq 1,
\end{equation}
with $X^0_{h}=P_hX_0$.

If $\varphi$ is a twice differentiable real functional on $L^2({\mathcal O})$, not necessarily bounded and with not necessarily  bounded first derivative but with bounded second derivative and $A^{(\nu-\frac{1}{\rho})/2}Q^{1/2}$ is a Hilbert-Schmidt operator on $H$, then our main result can be stated as follows. Denoting the expectation by $\E$, the the so-called weak error can be estimated as
\begin{equation}\label{eq:weak}
| \E \varphi (X_h^ N) - \E \varphi(X(T)) | \leq C\ln\left(\frac{T}{h^{2/\rho}+\Delta t}\right) (\Delta t^{{\rho}\nu} + h^{2\nu} ),
\end{equation}
\noindent where $C$ may depend on $T$, $\varphi$ and the initial condition $X(0)$ but not on $h$ and $N$. Hence even in the presence of memory, similarly to parabolic and hyperbolic stochastic equations \cite{AL,debussche_printems,KLLweak,fullweak,Lind,wang}, the weak order is essentially twice the strong order, where the latter was studied in \cite{kovacsprintems}. Note that in \eqref{eq:weak} we may even allow the test function $\varphi(x)=\|x\|^2$.

In particular, when $Q=I$ (white noise), $\nu$ has to be chosen in such a way that $A^{(\nu -1/\rho)/2}$ is an Hilbert-Schmidt operator on $H$. Taking the asymptotics of the eigenvalues of the Laplacian into account we must have $\nu < 1/\rho - d/2$. This yields $d=1$ and a rate of convergence in time of $(1 - \rho/2)_-$ and in space of $(2/\rho - 1)_-$.

A popular method used in the study the weak convergence of approximations of stochastic equations relies on the associated Kolomogorov equation (see \cite{kloeden,milsteintretyakov,talay}). The global weak error $\E \varphi (X_h^N) - \E \varphi(X(T))$ is usually decomposed in a sum of local weak errors which are then expressed using the solution of the associated Kolmogorov equation.

Unfortunately, the presence of a  non-local term in time in \eqref{eq:stovolterra} prevents us to use the same method directly because the process $\{X(t)\}_{t\in [0,T]}$ is not Markovian and hence no Kolmogorov equation is associated with \eqref{eq:stovolterra}. However, since the equation is linear and the non-local term is in the drift part, we use the same kind of method as in \cite{debussche_printems,KLLweak,fullweak,Lind} to remove the drift and obtain an equation which has a Markovian solution. Hence there is an associated Kolmogorov's equation with no drift but with a time-dependent covariance operator.

The outline of the paper is as follows. In Section 2 we introduce basic notations and the main assumptions on $b$, together with the existence, uniqueness and regularity properties of $\{X(t)\}_{t\in [0,T]}$. In Section 3 we recall and prove some deterministic estimates concerning the solutions of \eqref{eq:stovolterra} and \eqref{eq:full_scheme} with $Q=0$. In particular, we recall the discrete mild formulation
of \eqref{eq:full_scheme} and establish its regularizing properties (Theorem \ref{thm:det}). In Section 4 we introduce the results which are needed
 in order to accommodate random initial data and unbounded test functions $\varphi$. This section ends with a representation formula of the weak error (Theorem \ref{main}) which symmetrize the role played by the discrete and the continuous solutions. Finally, in Section 5, we state and prove the main convergence result (Theorem \ref{theo:weak}).

\section{Preliminairies}\label{ch:pre}
In this section we introduce notation and collect some preliminary results. We also state the hypothesis on the convolution kernel $b$.
We denote the set of bounded linear operators on $H$ by $\mathcal{B}(H)$ endowed with the usual norm $\|\cdot\|_{\mathcal{B}(H)}$, where we drop the subscript from the norm if it is clear from the context. Let $\HS$ denotes the \textit{Hilbert-Schmidt}
operators on $H$; that is, $T\in \HS$ if $T$ is linear and for an orthonormal basis (ONB) $\{e_k\}_{k\in \mathbb{N}}$ of $H$
$$
\|T\|_{\HS}^2=\sum_{k\in \mathbb{N}}\|Te_k\|^2<\infty.
$$
\noindent In this case the sum is independent of the ONB. If a linear operator $T$ on $H$ can be written as
$$
Tx=\sum_{k\in \mathbb{N}}(x,a_k)b_k,~x\in H,
$$
with $\sum_{k\in \mathbb{N}}\|a_k\|\,\|b_k\|<\infty$, then $T$ is called a \textit{trace-class} operator. The trace-norm of $T$ is then defined
as
$$
\|T\|_{\Tr}=\inf\{\sum_{k\in \mathbb{N}}\|a_k\|\,\|b_k\|:~Tx=\sum_{k\in \mathbb{N}}(x,a_k)b_k,~x\in H\}.
$$
If $T$ is a trace class operator then for any ONB, the trace of $T$ is defined as
$$
\Tr(T)=\sum (Te_k,e_k)
$$
is finite and the sum is independent of the ONB. Both Hilbert-Schmidt and trace class operators are bounded.
If $T\ge 0$ is a symmetric trace class operator then, $\Tr(T)=\|T\|_{\Tr}$. It is well-known that if $T_1,T_2\in \HS$, then $T_1T_2$ is trace class and
\begin{equation}\label{eq:trhs}
|\Tr(T_1T_2)|\le \|T_1\|_{\HS}\|T_2\|_{\HS}.
\end{equation}
Furthermore, if $T\in \HS$ and $S\in \mathcal{B}(H)$, then $TS$ and $ST$ are in $\HS$ and
\begin{equation}\label{eq:bdhs}
\max\{\|TS\|_{\HS},\|ST\|_{\HS}\}\le \|T\|_{\HS}\|S\|_{\mathcal{B}(H)}.
\end{equation}
For $1\le p<\infty$ we denote by $L^p(\Omega,H)$ the space of $H$-valued random variables $X$ such that $$\|X\|_{L^p(\Omega,H)}=(\mathbf{E}\|X\|^p)^{1/p}<\infty.$$

It is well known that our assumptions on $A$ and on the spatial domain $\mathcal{O}$ implies the existence of a sequence of nondecreasing positive real numbers $\{\lambda_k\}_{k\geq 1}$ and an orthonormal basis
$\{e_k\}_{k\geq 1}$ of $H$ such that
\begin{equation}\label{eq:spectral}
Ae_k = \lambda_k e_k, \quad \lim_{k\rightarrow +\infty} \lambda_k = +\infty.
\end{equation}

Next, we define, by means of the spectral decomposition of $A$, the fractional powers $A^s$ of $A$ for $s\in \mathbb{R}$. That is,
for $s>0$ we set
\begin{equation}\label{eq:fp}
A^sx=\sum_{k\ge 1}\lambda_k^s(x,e_k)e_k,
\end{equation}
with domain $D(A^s)$ being all $x\in H$ for which the sum converges in $H$. In particular, $D(A^0)=H$. For $s<0$ we define
$A^sx$ as in \eqref{eq:fp} for all $x\in H$ and $D(A^s)$ to be the completion of $H$ with respect to
the norm of $\|x\|_s=\|A^sx\|$.

Next we state the main hypothesis on the convolution kernel $b$.
\begin{hypothesis}\label{hyp:b}
The kernel $0\neq b\in L^1_{loc}(\R_+)$, is $3$-monotone; that is, $b$, $-\dot b$ are nonnegative, nonincreasing, convex, and $\lim_{t\to \infty}b(t)=0$. Furthermore,
\begin{equation}\label{eq:sector}
\rho  := 1 + \frac{2}{\pi}\sup \{ | \mathrm{arg} \, \widehat b(\lambda) |, \; \Re\lambda >0 \} \in (1,2).
\end{equation}
\end{hypothesis}
In the special case of the Riesz kernel given in \eqref{eq:beta} one can easily show that $\rho$ in the exponent coincides with the one defined in \eqref{eq:sector}.
In order to obtain non-smooth data estimates for the deterministic equation we need the following additional hypothesis.
\begin{hypothesis}\label{hyp:b-anal}
The Laplace transform $\widehat{b}$ of $b$ can be extended to an analytic function in a sector $\Sigma_\theta$ with $\theta>\pi/2$ and $|\widehat{b}^{(k)}(z)|\le C|z|^{1-\rho-k}$, $k=0,1$, $z\in \Sigma_\theta$.
\end{hypothesis}
In the sequel we discuss properties of the solution of \eqref{eq:stovolterra}.
The weak solution of \eqref{eq:stovolterra} is a mean-square continuous $H$-valued process satisfying
\begin{equation*}
  \left( X(t), \eta \right)
 + \int_0^t \int_0^rb(r-s)\left( X(s), A^*\eta\right)\,\dd s\,\dd r
=\left( X_0,\eta\right)
 + \int_0^t \left(\eta, \,\dd W^{Q}(s) \right),
\end{equation*}
for all $\eta \in D(A^*)$ almost surely for all $t\in [0,T]$. Under Hypothesis \ref{hyp:b}, if $W^Q\equiv 0$; that is, the deterministic case, then there exists a resolvent family $\{S(t)\}_{t\geq 0} \subset {\mathcal B}(H)$ which is strongly continuous for $t\ge 0$, differentiable for $t>0$ and uniformly bounded by 1, see \cite[Corollary 1.2 and Corollary 3.3]{pruss}. The unique weak solution in the deterministic case is given by
$X(t) = S(t)X_0,~ t\in [0,T].$

The next result, which can be found in \cite{CDaPP} and \cite{kovacsprintems} summarizes the existence, uniqueness and regularity of weak solutions of \eqref{eq:stovolterra}.
\begin{proposition}\label{prop:eur}
Let $b$ satisfy Assumption \ref{hyp:b} and let $\|A^{(\nu-\frac{1}{\rho})/2}Q^{\frac12}\|_{\HS}<\infty$ and $A^{\frac{\nu}{2}}X_0\in L^2(\Omega,H)$ for some $\nu\ge 0$. Then \eqref{eq:stovolterra} has a unique weak solution given by the variation of constants formula
\begin{equation}\label{eq:vcf}
X(t)=S(t)X_0+\int_0^tS(t-s)\,\dd W^Q(s),~t\ge 0,
\end{equation}
with $\|A^{\frac{\nu}{2}}X(t)\|_{L^2(\Omega,H)}\le C$, for some $C>0$ and for all $t\ge 0$. Furthermore, $X$ has a version which is H\"older continuous of order less than $\min(\frac12,\frac{\rho \nu}{2})$.
\end{proposition}
\begin{remark}\label{rem:sq}
In particular, the stochastic integral in \eqref{eq:vcf} makes sense since $SQ^{1/2} \in L^2((0,T),\HS)$ (see the proof of \cite[Theorem 3.6]{kovacsprintems}).
\end{remark}
\section{Deterministic estimates}
In the following proposition we collect some of the smoothing properties of $\{S(t)\}_{t\ge 0}$ and $\{S_h(t)\}_{\ge 0}$,
where $\{S_h(t)\}_{t\ge 0}\subset V_h$ is the resolvent family of the deterministic equation
$$
\dot{u}_h(t)+\int_0^tb(t-s)A_hu_h(s)\,\dd s=0,~t>0, \quad u_h(0)=P_hu_0.
$$
We would like to note that in this paper the constant $C$ denotes a generic nonnegative constant that does not depend on the parameters $h,k,t,\Delta t$ and is not necessarily the same at every occurrence.
\begin{proposition}\label{prop:sm}
If $b$ satisfies Hypothesis \ref{hyp:b}, then the following estimates hold for the resolvent families $\{S(t)\}_{t\ge 0}$ and $\{S_h(t)\}_{t\ge 0}$ for some $C>0$.
\begin{itemize}
\item[(i)] $\max\{\|A^{\nu/2}S(t)\|,\|A^{\nu/2}_hS_h(t)P_h\|\}\le C \, t^{-\nu\rho/2}$, $0\le \nu \le 2/\rho$, $t>0$, $h>0$;
\item[(ii)]$\|\dot{S}(t)\|\le C \, t^{-1}$, $t>0$;
\item[(iii)]$\int_0^t\|A_h^{1/(2\rho)}S_h(s)P_hx\|^2\,\dd s\le C \, \|x\|^2$,  $t>0$, $h>0$.
\end{itemize}
\end{proposition}
\begin{proof}
The statements for $\{S(t)\}_{t\ge 0}$ are shown in \cite[Proposition 2.4]{kovacsprintems}. The estimate for
$\{S_h(t)\}_{t\ge 0}$ in $(i)$ can be proved exactly the same way while $(iii)$ is shown in the proof of \cite[Lemma 3.1.]{kovacsprintems}.
\end{proof}
In the sequel we derive the relevant deterministic error estimates.
Using the $z$-transform, it is shown \cite{kovacsprintems} that  the solution $X_h^n$ of \eqref{eq:full_scheme} can be written using a discrete constant variation formula as
\begin{equation}\label{eq:DVCF}
X_h^n = B_{n,h} P_h x + \sum_{k=0}^{n-1} B_{n-k,h} P_h w_{k+1},
\end{equation}
where, $B_{0,h}=I$ and
\begin{equation}\label{eq:bnh}
B_{k,h}P_h x=\int_0^\infty S_h(\Delta t s)P_h x\frac{e^{-s}s^{k-1}}{(k-1)!}\,\dd s\text{ for }k\ge 1.
\end{equation}

Let $\sigma(t):=\lceil \frac{t}{\Delta t}\rceil$ and define the piecewise constant operator function
$$
\tilde{B}_{h,N}(t):=B_{\sigma(t),h}P_h,\quad 0\le t\le T.
$$
\begin{theorem}\label{thm:det}
If $b$ satisfies Hypothesis \ref{hyp:b} and \ref{hyp:b-anal}, then
 the following estimates hold for some $C>0$ where $E_{h,N}(t)=\tilde{B}_{h,N}(t)-S(t)$, $N\Delta t=T$ and $h>0$.
\begin{equation}\label{hold}
\|S(t)-S(s)\|_{{\mathcal B}(H)}\le C \, s^{-\alpha}|t-s|^{\alpha},~0\le \alpha\le 1,~0<s\le t;
\end{equation}
\begin{equation}\label{dsmooth}
\|A^{\frac{\nu}{2}}\tilde{B}_{h,N}(t)\|\le C \, t^{-\rho \nu/{2}},~0\le \nu \le \frac{1}{\rho},~0< t\le T;
\end{equation}
\begin{equation}\label{err}
\|E_{h,N}(t)\|\le C \, t^{-\rho \nu}(\Delta t^{\rho \nu}+h^{2\nu}),~0\le \nu\le \frac{1}{\rho},~0< t\le T;
\end{equation}
\begin{equation}\label{serr}
\|A^{\frac{1/\rho-\nu}{2}}E_{h,N}(t)\|\le  C \, t^{-\frac12-\frac{\rho \nu}{2}}(\Delta t^{\rho \nu}+h^{2\nu}),~0\le \nu \le \frac{1}{\rho},~0< t\le T.
\end{equation}
\end{theorem}
\begin{proof}
It follows from \cite[Corollary 3.3]{pruss} that $\|\dot{S}(t)x\|\le C t^{-1}\|x\|$ for all $x\in H$ and $t>0$. Thus, for $0<s\le t$, we have
$$
\|S(t)x-S(s)x\|\le \int_s^t\|\dot{S}(r)x\|\,\dd r\le C\|x\|\int_s^t r^{-1} \,\dd r \le C\|x\| s^{-1}|t-s|.
$$
Since we also have that $\|S(t)x-S(s)x\|\le 2 \|x\|$, the inequality in \eqref{hold} follows.

To show \eqref{dsmooth}, first note that it follows from \eqref{eq:bnh} and Proposition \ref{prop:sm} (i) with $\nu=0$ that
\begin{equation}\label{bcon}
\|B_{k,h}\|\le C,\text{ for all}~k\ge 1,~h>0.
\end{equation}
 From \eqref{eq:bnh} and Proposition \ref{prop:sm} (i) with $\nu=2/\rho$ we conclude that, for $k\ge 2$ and $h>0$,
\begin{equation}\label{eq:k2}
\begin{aligned}
&\|A_h^{1/\rho}B_{k,h}P_hx\|\le C\|x\| (\Delta t)^{-1} \int_0^\infty \frac{e^{-s}s^{k-2}}{(k-1)!}\,\dd s\\
&\quad = C\|x\| ((k-1)\Delta t)^{-1} \int_0^\infty\frac{e^{-s}s^{k-2}}{(k-2)!}\,\dd s=C\|x\|t^{-1}_{k-1}
=C\|x\|\frac{k}{k-1}t^{-1}_{k}\le C\|x\|t^{-1}_{k}.
\end{aligned}
\end{equation}
For $k=1$, by H\"older's inequality and Proposition \ref{prop:sm} (iii), we have
\begin{equation}\label{eq:k1}
\begin{aligned}
\|A_h^{1/(2\rho)}B_{1,h}P_hx\|&\le \int_0^{\infty}\|A_h^{1/(2\rho)}S_h(\Delta t s)x\|e^{-s}\,\dd s\\
&\le
C \left(\int_0^{\infty}\|A_h^{1/(2\rho)}S_h(\Delta t s)P_hx\|^2\,\dd s\right)^{1/2}\le C(\Delta t)^{-1/2}\|x\|.
\end{aligned}
\end{equation}
By interpolation, using \eqref{bcon}, \eqref{eq:k1} and \eqref{eq:k2} we conclude that
\begin{equation*}
\|A_h^{\frac{\nu}{2}}B_{h,k}P_h\|\le Ct_k^{-\frac{\rho \nu}{2}},~0\le \nu \le \frac{1}{\rho},~k\ge 1,~h>0.
\end{equation*}
Since for $\delta\in [0,1/2]$ and $v_h\in S_h$ we have that $\|A^\delta v_h\|\le \|A_h^{\delta}v_h\|$ it also follows that
\begin{equation*}
\|A^{\frac{\nu}{2}}B_{h,k}P_h\|\le Ct_k^{-\frac{\rho \nu}{2}},~0\le \nu \le \frac{1}{\rho},~k\ge 1,~h>0.
\end{equation*}
Finally,  for $t\in (t_{j-1},t_j]$, $j\ge 1$, we see that
$$
\|A^{\frac{\nu}{2}}\tilde{B}_{h,N}(t)\|=\|A^{\frac{\nu}{2}}B_{h,j}P_h\|\le Ct_j^{-\frac{\rho \nu}{2}}\le Ct^{-\frac{\rho \nu}{2}},~0\le \nu \le \frac{1}{\rho},~h>0.
$$
and the proof of \eqref{dsmooth} is complete.

Next we prove \eqref{err}. First, we write
$$
\|B_{h,k}P_h -S(t_k)\|\le\|B_{h,k}-S_h(t_k)\|+\|S_h(t_k)-S(t_k)\|:=e_1+e_2.
$$
It is show in \cite{Lubich_et_al1996} that if $b$ satisfies Hypothesis \ref{hyp:b-anal}, then
$$
e_1\le C t_k^{-1}\Delta t \text{ and }e_2\le C t_k^{-\rho}h^2,~k\ge 1,~h>0.
$$
Furthermore, we also have that $\max\{e_1,e_2\}\le C$ by Proposition \ref{prop:sm} with $\nu=0$, and thus
\begin{equation}\label{eq:lst}
\|B_{h,k}P_h -S(t_k)\|\le C t_k^{-\rho \nu}(\Delta t^{\rho \nu}+h^{2\nu}), ~0\le \nu \le 1/\rho,~k\ge 1,~h>0.
\end{equation}
Next, for $t\in (t_{k-1},t_k]$, $k\ge 1$, we have by \eqref{hold} and \eqref{eq:lst}, that
\begin{align*}
\|E_{h,N}(t)\|&\le \|B_{h,k}P_h-S(t_k)\|+\|S(t_k)-S(t)\|\\
&\le C t_k^{-\rho \nu}(\Delta t^{\rho \nu}+h^{2\nu})+Ct_k^{-\rho \nu}\Delta t^{\rho s}\\
&\le Ct^{-\rho \nu}(\Delta t^{\rho \nu}+h^{2\nu}), ~0\le \nu \le 1/\rho,~k\ge 1,~h>0.
\end{align*}
which finishes the proof of \eqref{err}.

Finally, by interpolation, for $0\le \alpha\le 1/(2\rho)$ have that
\begin{align*}
\|A^\alpha E_{h,N}(t)\|&\le \|E_{h,N}(t)\|^{1-2\rho \alpha}\|A^{1/(2\rho)}E_{h,N}(t)\|^{2\rho \alpha}\\
&\le  \|E_{h,N}(t)\|^{1-2\rho \alpha}\left(\|A^{1/(2\rho)}S(t)\|^{2\rho \alpha}+\|A^{1/2\rho}\tilde{B}_{h,N}(t)\|^{2\rho \alpha}\right).
\end{align*}
Setting $\alpha=\frac{1/\rho-\nu}{2}$, $0\le \nu \le 1/\rho$, and using Proposition \ref{prop:sm} (i), \eqref{dsmooth} and \eqref{err} all with $\nu=1/\rho$ the estimate in \eqref{serr} follows.
\end{proof}

\section{Error representation}


Our main assumptions concerning $\varphi$, depending on the initial data, are
\begin{equation}\label{eq:phi0}
\varphi \in C(H,\mathbb{R}), D\varphi \in C(H,H)\text{ and }D^2\varphi \in C_b(H,\mathcal{B}(H)).
\end{equation}
or
\begin{equation}\label{eq:phi1}
\varphi \in C(H,\mathbb{R}), D\varphi \in C_b(H,H)\text{ and }D^2\varphi \in C_b(H,\mathcal{B}(H)),
\end{equation}
where $C(X,Y)$ and $C_b(X,Y)$ denote the space of continuous resp. continuous and bounded functions $f:X\to Y$ and $D$ denotes the Fr\'echet derivative.
The next lemma and its corollary are needed in the proof of Theorem \ref{main} to accommodate  random initial data and test functions $\varphi$ that are unbounded with possibly an unbounded first derivative. For bounded test functions $\varphi$ the next result can be found, in for example, \cite[Proposition 1.12]{DPZ}.

\begin{lemma}\label{lem:ind}
Let $\varphi:H\to \mathbb{\R}$ be measurable such that $|\varphi(x)|\le p_N(\|x\|)$ where $p_N$ is a real polynomial of degree $N$ with non-negative coefficients. Let $(\Omega,\mathcal{F},P)$ be a probability space and $\mathcal{G}\subset \mathcal{F}$ is a sub sigma-algebra of $\mathcal{F}$. Let $\xi_1,\xi_2\in L^N(\Omega,H)$ be $H$-valued random variables such that $\xi_1$ is $\mathcal{G}$-measurable and $\xi_2$ is independent of $\mathcal{G}$.  If we define $u:H\to \mathbb{R}$ by
$u(x)=\mathbf{E}(\varphi(x+\xi_2))$, $x\in H$, then, almost surely, $u(\xi_1)=\mathbf{E}(\varphi(\xi_1+\xi_2)|\mathcal{G})$.
\end{lemma}
\begin{proof}
Define $\varphi_n(x)=\varphi(\xi_{B_n(0)}(x)x)$ where $\xi_{B_n(0)}$ is the characteristic function of the closed unit ball around $0$ with radius $n$.
We clearly have that $\varphi_n(x)\to \varphi(x)$ for all $x\in H$. Furthermore, $|\varphi_n(x)|\le p_N(\|x\|)$ for all $n\in \mathbb{N}$ and $x\in H$. Therefore, if $\eta\in L^N(\Omega, H)$, then by the dominated convergence theorem $\varphi_n(\eta)\to \varphi(\eta)$ in $L^1(\Omega,\mathbb{R})$. Let $x\in H$ and define $u(x):=\mathbf{E}(\varphi(x+\xi_2))$ and $u_n(x):=\mathbf{E}(\varphi_n(x+\xi_2))$. If we take $\eta:=x+\xi_2$, then, for all $x\in H$,
$$
|u_n(x)-u(x)|\le |\mathbf{E}(\varphi_n(\eta)-\varphi(\eta))|\le \|\varphi_n(\eta)-\varphi(\eta)\|_{L^1(\Omega,\mathbb{R})}\to 0
$$
as $n\to \infty$. We also have that
\begin{multline*}
|u_n(x)|\le \mathbf{E}|(\varphi_n(x+\xi_2))|\le \mathbf{E}( p_N(\|x+\xi_2\|))\\
\le C (p_N(\|x\|)+\mathbf{E}(p_N(\|\xi_2\|))\le C(p_N(\|x\|)+\|\xi_2\|_{L^N(\Omega,H)}),
\end{multline*}
and hence
$$
|u_n(\xi_1)|\le C(p_N(\|\xi_1\|)+\|\xi_2\|_{L^N(\Omega,H)})\in L^1(\Omega; \mathbb{R}).
$$
Therefore,
\begin{equation}\label{eq:coxi}
u_n(\xi_1)\to u(\xi_1) \text{ in }L^1(\Omega,\mathbb{R})
\end{equation}
as $n\to \infty$ by dominated convergence. Since $\phi_n$
is a bounded and measurable function it follows from \cite[Proposition 1.12]{DPZ} that $u_n(\xi_1)=\mathbf{E}(\varphi(\xi_1+\xi_2)|\mathcal{G})$. By taking $\eta=\xi_1+\xi_2$ it follows as above that  $\varphi_n(\xi_1+\xi_2)\to \varphi(\xi_1+\xi_2)$ in $L^1(\Omega, \mathbb{R})$ and thus
by the dominated convergence theorem for conditional expectations we conclude that $$u_n(\xi_1)=\mathbf{E}(\varphi_n(\xi_1+\xi_2)|\mathcal{G})\to \mathbf{E}(\varphi(\xi_1+\xi_2)|\mathcal{G})\text{ in }L^1(\Omega, H)$$ as $n\to \infty$  which finishes the proof in view of \eqref{eq:coxi}.
\end{proof}

For any $x\in H$ and $t\in [0,T]$, we define
$$
Z(T,t,x):=x+\int_t^TS(T-s)\,\dd W^Q(s).
$$
 The above stochastic integral makes sense since $S Q^{1/2} \in L^2((0,T),\HS)$ (see Remark \ref{rem:sq}). Let $\varphi$  satisfy \eqref{eq:phi0} or \eqref{eq:phi1}, and define
\begin{equation} \label{eq:defu}
u(x,t):=\mathbf{E}(\varphi(Z(T,t,x))),~x\in H,~t\in [0,T].
\end{equation}
Since $\mathrm{Tr}(S(T-\cdot)QS(T-\cdot)^*) \in L^1(0,T)$ (see Remark \ref{rem:sq}) and $D^2\varphi \in C_b(H,{\mathcal B}(H))$, it
is well known that $u$ is a solution of the following backward Kolmogorov equation
\begin{equation}\label{eq:kolmogorov}
u_t(x,t) + \frac 12\, \mathrm{Tr}\Big (u_{xx}(x,t) S(T-t)Q S(T-t)^*\Big) = 0, \quad x\in H, \; t\in [0,T),
\end{equation}
\noindent with the terminal condition $u(T,x) = \varphi(x)$, $x\in H$.

\begin{corollary}\label{cor:u}
Let $\xi$ be $\mathcal{F}_t$-measurable where $\{\mathcal{F}_t\}_{t\ge 0}$ is the normal filtration generated by $W$.
Let  $\varphi$ satisfy \eqref{eq:phi0} and $\xi\in L^2(\Omega,H)$ or let $\varphi$ satisfy \eqref{eq:phi1} and $\xi\in L^1(\Omega,H)$.
Let $u$ defined by (\ref{eq:defu}). Then
\begin{equation*}
u(\xi,t)=\mathbf{E}(\varphi(Z(T,t,\xi))|\mathcal{F}_t),~t\in [0,T].\\
\end{equation*}
\end{corollary}
\begin{proof}
The statement follow from Lemma \ref{lem:ind} with $\xi_1=\xi$ and $\xi_2=\int_t^TS(T-s)\,\dd W^Q(s)$ noting that
$\xi_2\in L^2(\Omega,H)\subset L^1(\Omega,H)$ as, by It\^o's Isometry,
$$
\mathbf{E}\|\xi_2\|^2=\int_t^T\|S(T-s)Q^{\frac{1}{2}}\|^2_{\HS}\,\dd s\le \int_0^T\|S(t)Q^{\frac{1}{2}}\|^2_{\HS}\,\dd t<\infty.
$$
\end{proof}
We quote the following It\^o's formula from \cite{Brez}.
\begin{proposition}[It\^o's formula]\label{prop:ito}
Let $f:[c,d)\times H\to\mathbb{R}$, $0\le c<d\le \infty,$ such that $f,\partial_t f, \partial_x f$ and $\partial^2_{xx} f$ are continuous on $[c,d)\times H$ with values in the appropriate spaces. Let $a\in L^1_{loc}(\Omega\times (c,d);H)$ and $\xi Q^{1/2}\in L^2_{loc}(\Omega\times(c,d),\HS)$ and
$$
X(t)=X(c)+\int_c^ta(s)\,\dd s+\int_c^t \xi(s)\,\dd W^Q(s),~ t\in [c,d).
$$
Then, for all $t\in [c,d)$, almost surely,
\begin{align*}
&f(t,X(t))-f(c,X(c))=\int_c^t \partial_t f(s,X(s))\,\dd s + \int_c^t(\partial_x f(s,X(s)), a(s))\,\dd s\\
&+\int_c^t(\partial_x f(s,X(s)), \xi(s)\,\dd W^Q(s))+\frac{1}{2}\int_c^t \Tr(\partial^2_{xx}f(s,X(s))\xi(s) Q\xi^*(s))\,\dd s.
\end{align*}
\end{proposition}

The proof of the main approximation result of the paper relies on the ability to compare the laws of two different It\^{o} processes of the form
$$
Y(t):=Y(0)+\int_{0}^tS(T-s)\,\dd W^Q(s)
$$
and
$$
\T{Y}(t):=\T{Y}(0)+\int_{0}^t\T{S}(T-s)\,\dd W^Q(s),
$$
\noindent where $\{\tilde S(t)\}_{t>0}$ denotes another family of bounded operators on $H$ such that $\tilde SQ^{1/2} \in L^2((0,T),\HS)$.
We have the following general error formula for
\begin{equation}\label{eq:weakErr}
e(T)={\bf E}\left(\varphi(\T{Y}(T))-\varphi({Y(T)})\right).
\end{equation}
\begin{theorem}\label{main}
 If $\varphi$ satisfies \eqref{eq:phi0} and $Y(0),\T{Y}(0)\in L^2(\Omega,H)$ or  $\varphi$ satisfies \eqref{eq:phi1} and $Y(0),\T{Y}(0)\in L^1(\Omega,H)$ and $SQ^{1/2},\tilde{S}Q^{1/2}\in L^{2}((0,T),\HS)$, then $Y$ and $\tilde{Y}$ are well-defined and the weak error $e(T)$ in
\eqref{eq:weakErr} has the representation
\begin{equation*}
e(T) ={\bf E}\big(u(\tilde{Y}(0),0)-u(Y(0),0)\big)
+\tfrac{1}{2}{\bf E}\int_0^T
\Tr\Big( u_{xx}(\tilde{Y}(t),t)\mathcal{O}(t)\Big)\,\dd t,
\end{equation*}
where
\begin{equation*}
\mathcal{O}(t)=\big(\tilde{S}(T-t)+S(T-t)\big)Q\big(\tilde{S}(T-t)-S(T-t)\big)^*,
\end{equation*}
or
\begin{equation*}
\mathcal{O}(t)=\big(\tilde{S}(T-t)-S(T-t)\big)Q\big(\tilde{S}(T-t)+S(T-t)\big)^*.
\end{equation*}
\end{theorem}
\begin{proof}
The proof is analogous to the semigroup case in \cite[Theorem 3.1]{KLLweak} and \cite[Theorem 3.1]{fullweak} using Proposition \ref{prop:ito}, Corollary \ref{cor:u} and the Kolmogorov's equation (\ref{eq:kolmogorov}).
\end{proof}

\section{The convergence result}
In this section we apply Theorem \ref{main} to the approximation scheme \eqref{eq:full_scheme}. In order to do so we set
$$
Y(t)=S(T)X_0+\int_0^tS(T-s)\,\dd W^Q(s)
$$
and
$$
\tilde{Y}(t)=\tilde{B}_{h,N}(T)X_0+\int_0^t\tilde{B}_{h,N}(T-s)\,\dd W^Q(s).
$$
Note that $Y(T)=X(T)$ and $\tilde{Y}(T)=X^N_h$.
Our main result is stated below.
\begin{theorem}\label{theo:weak}
Let $T>0$, $N\geq 1$ an integer and $\Delta t = T/N$. For any $h>0$, let $\{X^n_h\}_{0\leq n \leq N}$ be defined by \eqref{eq:full_scheme} and let $\{X(t)\}_{t\in[0,T]}$ be the unique weak solution \eqref{eq:vcf} of \eqref{eq:stovolterra}. Let $\varphi$ satisfy \eqref{eq:phi0} and suppose that $X_0\in L^2(\Omega,H)$ or let $\varphi$ satisfy \eqref{eq:phi1} and suppose $X_0\in L^1(\Omega,H)$. If $\|A^{\frac{\nu-1/\rho}{2}}Q^\frac12\|^2_{\HS}<\infty$, $0\le \nu\le 1/\rho$, then there exists a constant $C = C(T,\nu,\varphi,X_0)>0$ which does not depend on $h$ and $N$ such that for $h^{2/\rho}+\Delta t<T$,
\begin{equation}\label{eq:weakorder}
|\E \, \varphi(X^N_h) - \E \, \varphi(X(T))| \leq C\ln\left(\frac{T}{h^{2/\rho}+\Delta t}\right) (\Delta t^{\rho\nu} + h^{2\nu} ).
\end{equation}
\end{theorem}
\begin{proof}
We use Theorem \ref{main} with $Y(0)=S(T)X_0$, $\tilde{Y}(0)=\tilde{B}_{h,N}(T)X_0$, and $\tilde{S}(t)=\tilde{B}_{h,N}(t)$.
Using \eqref{eq:bdhs} and \eqref{dsmooth} with $\nu=0$, we have that $$\|\tilde{B}_{h,N}(t)\|_{\HS}=\|P_h\tilde{B}_{h,N}(t)\|_{\HS}\le C\|P_h\|_{\HS},~t\in (0,T).$$ Therefore, as $V_h$ is finite dimensional, it follows that $\tilde{S}Q^{1/2}\in L^2((0,T), \HS)$.
Furthermore,
$$
\int_0^T\|S(t)Q^{\frac12}\|^2_{\HS}\,\dd t\le C \|A^{-\frac{1}{2\rho}}Q^\frac12\|^2_{\HS}\le C \|A^{-\frac{\nu}{2}}\|_{\mathcal{B}(H)}\|A^{\frac{\nu-1/\rho}{2}}Q^\frac12\|^2_{\HS}<\infty,
$$
where the first inequality is shown in the proof of \cite[Theorem 3.6]{kovacsprintems} and the second inequality follows from \eqref{eq:bdhs}.
Thus, Theorem \ref{main} is applicable.

We estimate the trace term first. Using that the operators $A,\tilde{B}_{h,N}$, and $S$ are self-adjoint, and taking inequalities \eqref{eq:trhs} and \eqref{eq:bdhs} into account, we have
\begin{align}
&\Big|{\bf E}\int_0^T\Tr\Big(u_{xx}(\tilde Y(t),t) \nonumber\\
&\qquad\times [\tilde{B}_{h,N}(T-t)+S(T-t)]Q[\tilde{B}_{h,N}(T-t)-S(T-t)]^*\Big)\,\dd t\Big|\nonumber\\
&\quad=\Big|{\bf E}\int_0^T\Tr\Big(u_{xx}(\tilde Y(t),t)[\tilde{B}_{h,N}(T-t)+S(T-t)]^*\nonumber\\
& \qquad \times A^{\frac{1/\rho-\nu}{2}}A^{\frac{\nu-1/\rho}{2}}Q^\frac12 Q^{\frac12}A^{\frac{\nu-1/\rho}{2}}A^{\frac{1/\rho-\nu}{2}}E_{h,N}(T-t)\Big)\,\dd t\Big|\nonumber\\
&\quad =\Big|{\bf E}\int_0^T\Tr\Big(u_{xx}(\tilde Y(t),t)(A^{\frac{1/\rho-\nu}{2}}[\tilde{B}_{h,N}(T-t)+S(T-t)])^*\nonumber\\
&\qquad\times A^{\frac{\nu-1/\rho}{2}}Q^\frac12 Q^{\frac12}A^{\frac{\nu-1/\rho}{2}}A^{\frac{1/\rho-\nu}{2}}E_{h,N}(T-t)\Big)\,\dd t\Big|\nonumber\\
&\quad\le {\mathbf E}\int_0^T \|u_{xx}(\tilde Y(t),t)(A^{\frac{1/\rho-\nu}{2}}[\tilde{B}_{h,N}(T-t)+S(T-t)])^*A^{\frac{\nu-1/\rho}{2}}Q^\frac12\|_{\HS}\nonumber\\
&\qquad\times\|Q^{\frac12}A^{\frac{\nu-1/\rho}{2}}A^{\frac{1/\rho-\nu}{2}}E_{h,N}(T-t)\|_{\HS}\,\dd t\nonumber\\
&\quad \le \sup_{(x,t)\in H\times [0,T]}\|u_{xx}(x,t)\|_{\mathcal{B}(H)}\|A^{\frac{\nu-1/\rho}{2}}Q^\frac12\|^2_{\HS}\\
&\qquad\times\int_0^T\|A^{\frac{1/\rho-\nu}{2}}(\tilde{B}_{h,N}(t)+S(t))\|_{\mathcal{B}(H)}
\|A^{\frac{1/\rho-\nu}{2}}E_{h,N}(t)\|_{\mathcal{B}(H)}\,\dd t .
\end{align}
Next we split the integral from $0$ to $\Delta t + h^{2/\rho}$ and from $h^{2/\rho}+\Delta t$ to $T$. Then, using Proposition \ref{prop:sm} (i) and \eqref{dsmooth},
\begin{align*}
&\int_0^{\Delta t + h^{2/\rho}}\|A^{\frac{1/\rho-\nu}{2}}(\tilde{B}_{h,N}(t)+S(t))\|_{\mathcal{B}(H)}
\|A^{\frac{1/\rho-\nu}{2}}E_{h,N}(t)\|_{\mathcal{B}(H)}\,\dd t \\
&\quad\le 2
 \int_0^{\Delta t + h^{2/\rho}}\left(\|A^{\frac{1/\rho-\nu}{2}}\tilde{B}_{h,N}(t)\|^2_{\mathcal{B}(H)}+\|A^{\frac{1/\rho-\nu}{2}}S(t)\|^2_{\mathcal{B}(H)}\right)\,\dd t\\
 &\quad\le  C\int_0^{\Delta t + h^{2/\rho}} t^{-1+\rho \nu}\,\dd t\le C (\Delta^{\rho \nu}+h^{2\nu}).
\end{align*}
Furthermore, by \ref{prop:sm} (i), \eqref{dsmooth} and \eqref{serr}, it follows that
\begin{align*}
&\int_{\Delta t + h^{2/\rho}}^T \|A^{\frac{1/\rho-\nu}{2}}(\tilde{B}_{h,N}(t)+S(t))\|_{\mathcal{B}(H)}
\|A^{\frac{1/\rho-\nu}{2}}E_{h,N}(t)\|_{\mathcal{B}(H)}\,\dd t\\
&\le C\int_{\Delta t + h^{2/\rho}}^T t^{-1/2+\rho \nu /2} (\Delta^{\rho \nu}+h^{2\nu})t^{-1/2-\rho\nu/2} \,\dd t= C \ln\left (\frac{T}{\Delta t + h^{2/\rho}}\right)(\Delta^{\rho \nu}+h^{2\nu}).
\end{align*}
This finishes the estimate of the trace term considering that
$$
\sup_{(x,t)\in H\times [0,T]}\|u_{xx}(x,t)\|_{\mathcal{B}(H)}\le \sup_{x\in H}\|D^2\varphi(x)\|_{\mathcal{B}(H)}.
$$
To estimate the initial error term first assume
 $\varphi$ satisfies \eqref{eq:phi0} and ${\bf E}\|X_0\|^2<\infty$. Note that under assumption \eqref{eq:phi0} it follows from Taylor's Formula
  that $$|\varphi(x)-\varphi(y)|\le \|D\varphi(y)\|\cdot\|x-y\|+C\|x-y\|^2,$$ where $C=\sup_{x\in H}\|D^2\varphi(x)\|_{\mathcal{B}(H)}$ and that $\|D\varphi(x)\|\le K(1+\|x\|)$ where $K=\max\{C,\|D\varphi(0)\|\}$. Therefore,
  $$
  |\varphi(x)-\varphi(y)|\le C (1+\|y\|)\cdot\|x-y\|+C\|x-y\|^2.
  $$
Then,  using the law of double expectations, and noting that $$X(T)=S(T)X_0+\int_0^TS(T-s)\,\dd W(s),$$
we have, using Corollary \ref{cor:u}, Proposition \ref{prop:eur} with $\nu=0$ and (\ref{err}), that
\begin{align*}
&\left|{\bf E}\left(u(\tilde{Y}(0),0)-u(Y(0),0)\right)\right|=\left|{\bf E}\left(u(\tilde{B}_{h,N}(T)X_0,0)-u(S(T)X_0,0)\right)\right|\\
&= \big|{\bf E}(\mathbf{E}((\varphi(\tilde{B}_{h,N}(T)X_0+\int_0^TS(T-s)\,\dd W^Q(s))\\
&\qquad\qquad\qquad\qquad\qquad\qquad\qquad-\varphi(S(T)X_0+\int_0^TS(T-s)\,\dd W^Q(s)))|\mathcal{F}_0))\big|\\
&\le C{\bf E}(\|\tilde{B}_{h,N}(T)X_0 -S(T)X_0\|\cdot(1+\|X(T)\|))+C{\bf E}(\|\tilde{B}_{h,N}(T)X_0 -S(T)X_0\|^2)\\
&\le CT^{-\rho \nu}(\Delta t^{\rho \nu}+h^{2\nu})(1+{\bf E}(\|X_0\|^2+\|X(T)\|^2))\\
&\quad\quad\quad\quad\quad\quad\quad\quad\quad\quad\quad\quad\quad\quad\quad\quad\quad\quad\quad\quad\quad+CT^{-2\rho \nu}(\Delta t^{2\rho \nu}+h^{4\nu}){\bf E}\|X_0\|^2.
\end{align*}
Finally, if $\varphi$ satisfies \eqref{eq:phi1} and ${\bf E}\|X_0\|<\infty$, then, again by Taylor's Formula,  $|\varphi(x)-\varphi(y)|\le C\|x-y\|$ with $C=\sup_{x\in H}\|D\varphi(x)\|$. Thus, similarly to the above calculation, we have that
\begin{align*}
&\left|{\bf E}\left(u(\tilde{Y}(0),0)-u(Y(0),0)\right)\right|=\left|{\bf E}\left(u(\tilde{B}_{h,N}(T)X_0,0)-u(S(T)X_0,0)\right)\right|\\
&= \big|{\bf E}(\mathbf{E}((\varphi(\tilde{B}_{h,N}(T)X_0+\int_0^TS(T-s)\,\dd W^Q(s))\\
&\qquad\qquad\qquad\qquad\qquad\qquad\qquad-\varphi(S(T)X_0+\int_0^TS(T-s)\,\dd W^Q(s)))|\mathcal{F}_0))\big|\\
&\le C{\bf E}(\|\tilde{B}_{h,N}(T)X_0 -S(T)X_0\|)\le CT^{-\rho \nu}(\Delta t^{\rho \nu}+h^{2\nu}){\bf E}\|X_0\|,
\end{align*}
and the proof is complete.
\end{proof}

\begin{remark} Below we give examples of the rate of convergence obtained in Theorem \ref{theo:weak} in some typical cases.\\
\noindent (i) If $Q=I$ (white noise), then, as mentioned in the introduction, we must have $d=1$ and the rate of weak convergence in time is $(1 - \rho/2)_-$ and in space it is $(2/\rho - 1)_-$.\\
\noindent (ii) If $Q$ is of trace class, then we may take $\nu=1/\rho$ and recover the finite dimensional order; that is, $1_-$ in time and $2_-$ in space.\\
\noindent (iii) Suppose that there exists some real numbers $\kappa$ and $\alpha>0$ such that $A^\kappa Q \in {\mathcal B}(H)$, $\mathrm{Tr}(A^{-\alpha}) < \infty$ and $\alpha - 1/\rho <\kappa \leq \alpha$. Then,
since
$$
\|A^{\frac{\nu-1/\rho}{2}}Q^\frac12\|^2_{\HS} \leq \|A^\kappa Q\|_{\mathcal B(H)} \mathrm{Tr}(A^{\nu-1/\rho - \kappa}),
$$
\noindent we recover a space weak order of convergence $(2/\rho - 2(\alpha -\kappa))_-$ and a time weak order of $(1-\rho(\alpha-\kappa)_-$. These are twice the strong orders (modulo the logarithmic term) respectively in space and in time found in \cite{kovacsprintems}.
\end{remark}
\noindent \textbf{Acknowledgement.} The authors wish to thank Dr. Fredrik Lindgren for valuable discussions.

\end{document}